\documentclass[reqno, a4paper]{amsart}
\usepackage{amsmath, amssymb,amsthm,amsfonts}

\usepackage[T1]{fontenc}
\usepackage[utf8]{inputenc}
\usepackage{enumerate}
\usepackage{color}
\usepackage{caption}
\usepackage{tikz}
\usetikzlibrary{patterns}

\newtheorem{theorem}{Theorem}

\newtheorem{lemma}[theorem]{Lemma}
\newtheorem{proposition}[theorem]{Proposition}

\theoremstyle{remark}
\newtheorem{remark}[theorem]{Remark}
\newtheorem{example}[theorem]{Example}

 \newcommand{\eps}{\varepsilon}

  \newcommand{\F}{\mathcal{F}}
 
 \newcommand{\M}{\mathsf{M}}
 
 \newcommand{\W}{\mathcal{W}}

 \renewcommand{\phi}{\varphi}

\newcommand{\cadlag}{c\`{a}dl\`{a}g}

\newcommand{\E}{\mathbb{E}}
\renewcommand{\P}{\mathbb{P}}
\newcommand{\N}{\mathbb{N}}

\newcommand{\D}{\mathbb{D}}

\newcommand{\R}{\mathbb{R}}

\newcommand{\law}{\mathrm{Law}}

\newcommand{\G}{\mathcal G}

\newcommand{\be}{\begin{equation}}
\newcommand{\ee}{\end{equation}}

\newcommand{\bea}{\begin{eqnarray}}
\newcommand{\bes}{\begin{subequations}}
\newcommand{\ees}{\end{subequations}}
\newcommand{\bgt}{\begin{gather}}
\newcommand{\egt}{\begin{gather}}
\newcommand{\eea}{\end{eqnarray}}
\newcommand{\beaa}{\begin{eqnarray*}}
\newcommand{\eeaa}{\end{eqnarray*}}

\renewcommand{\W}{{\mathbb W}}

\renewcommand{\eps}{\varepsilon}
\renewcommand{\epsilon}{\varepsilon}

\renewcommand{\L}{\mathsf{L}}

\newcommand{\fourIdx}[5]{%
\setbox1=\hbox{\ensuremath{^{#1}}}%
 \setbox2=\hbox{\ensuremath{_{#2}}}%
 \setbox5=\hbox{\ensuremath{#5}}%
 \hspace{\ifnum\wd1>\wd2\wd1\else\wd2\fi}%
 \ensuremath{\copy5^{\hspace{-\wd1}\hspace{-\wd5}#1\hspace{\wd5}#3}%
 _{\hspace{-\wd2}\hspace{-\wd5}#2\hspace{\wd5}#4}%
 }}

\numberwithin{equation}{section}
\numberwithin{theorem}{section}

\renewcommand{\subset}{\subseteq}

\newcommand{\Wa}{\textstyle W}

\usepackage{subcaption}
\usepackage{graphicx}

\begin{document}

\title{Root to Kellerer}

\author{Mathias Beiglb\"ock} \author{Martin
  Huesmann}  \author{Florian Stebegg}  \date{\today}
  \begin{abstract}
We revisit Kellerer's Theorem, that is, we show that for a family of real probability distributions $(\mu_t)_{t\in [0,1]}$ which increases in convex order there exists a Markov martingale $(S_t)_{t\in[0,1]}$ s.t.\ $S_t\sim \mu_t$.  

To establish the result, we observe that the set of martingale measures with given marginals carries a natural compact Polish topology. Based on a particular property of the martingale coupling associated to Root's embedding this allows for a relatively concise proof of Kellerer's theorem. 

We emphasize that many of our arguments are borrowed from Kellerer \cite{Ke72}, Lowther \cite{Lo07}, and Hirsch-Roynette-Profeta-Yor \cite{HiPr11,HiRo12}.


\end{abstract}
{\let\thefootnote\relax\footnote{ MB and FS  acknowledge support through FWF-projects P26736 and Y782-N25. MH   acknowledges support through  CRC 1060.
We also thank the Hausdorff Research Institute for Mathematics (HIM) for its hospitality in spring 2015 and Nicolas Juillet for many insightful comments. }}
\maketitle

\section{Introduction}

\subsection{Problem and basic concepts}
We consider couplings between probabilities $(\mu_t)_{t\in T}$ on the real line, where $t$ ranges over different choices of time sets $T$. Throughout we assume that all $\mu_t$ have a first moment. We represent these couplings as probabilities (usually denoted by $\pi$ or $\P$) on the canonical space $\Omega$ corresponding to the set of times under consideration. More precisely $\Omega $ may be $\R^T$
or the space $\mathcal D$ of \cadlag\ functions if $T=[0,1]$. In each case we will write $(S_t)$ for the canonical process and $\F=(\F_t)$ for the natural filtration.  $\Pi((\mu_t))$ denotes the set of probabilities $\P$ for which $S_t\sim_\P \mu_t$. $\M((\mu_t))$ will denote the subset of probabilities (``martingale measures'') for which $S$ is a martingale wrt\ $\F$ resp.\ the right-continuous filtration $\F^+=(\F^+_t)_{t\in [0,1]}$ in the case $\Omega=\mathcal D$. To have $\M((\mu_t))\neq \emptyset$
it is \emph{necessary} that $(\mu_t)$ increases in convex order, i.e.\ $\mu_s(\phi)\leq\mu_t(\phi)$ for all convex functions $\phi$ and $s\leq t$. This is an immediate consequence of Jensen's inequality. We denote the convex order by $\preceq.$

Our interest lies in the fact that this condition is also \emph{sufficient}, and we shall from now on assume that $(\mu_t)_{t\in T}$ increases in convex order, i.e.\ that $(\mu_t)_{t\in T}$ is a \emph{peacock} in the terminology of \cite{HiPr11,HiRo12}.  The proof that $\M((\mu_t)_{t\in T})\neq \emptyset$
gets increasingly difficult as we increase the cardinality of the set of times under consideration:

%
%

If  $T= \{1,2\}$, this follows from Strassen's Theorem (\cite{St65}) and we take this result for granted. The case $T= \{1, \ldots, n\}$ immediately follows by composition of one-period martingale measures $\pi_k\in \M(\mu_k,\mu_{k+1})$. 

If $T$ is not finite, the fact that $\M((\mu_t)_{t\in T})\neq \emptyset$  is less immediate  and to establish that  $\M((\mu_t)_{t\in T})$ contains a Markov martingale is harder still; these results were first proved by Kellerer in \cite{Ke72, Ke73} and now go under the name of Kellerer's theorem.  
 We recover these classical results in a framework akin to that of martingale optimal transport. 
 

\subsection{Comparison with Kellerer's approach}


Kellerer \cite{Ke73,Ke72} works with peacocks indexed by a general totally  ordered index set $T$ and the corresponding natural filtration $\F$. He establishes compactness of martingale measures on $\R^T$ which correspond to the peacock $(\mu_t)_{t\in T}$. Then Strassen's theorem allows him to show the existence of a martingale with given marginals $(\mu_t)_{t\in T}$ for general $T$. 

To show that $\M((\mu_t)_{t\in T})$ also contains a Markov martingale is more involved. On a technical level, an obstacle is
that the property of being a Markovian martingale measure is not suitably closed. 
 Kellerer  circumvents this difficulty based on a stronger notion of Markov kernel, the concept of \emph{Lipschitz} or \emph{Lipschitz-Markov kernels} on which all known proofs of Kellerer's Theorem rely. The key step to showing that $\M((\mu_t)_{t\in T})$ contains a Markov martingale is to establish the existence of a two marginal Lipschitz kernel. Kellerer achieves this  by showing that there are Lipschitz-Markov martingale kernels transporting a given distribution $\mu$ to the extremal points of the set ${\mu \preceq \nu}$ and subsequently
obtaining an appealing Choquet-type representation for this set.

Our aim is to give a compact, self contained presentation of Kellerer's result in a framework that can be useful for questions arising in martingale optimal transport\footnote{An early article to study this continuum time version of the  martingale optimal transport problem is the recent article \cite{KaTaTo15} of Kallblad, Tan, and Touzi.} for a continuum of marginals. 
While Kellerer is not interested in continuity properties of the paths of the corresponding martingales, it is favourable to work  in the more traditional setup of martingales with \cadlag\ paths to make sense of typical path-functionals (based on e.g.\ running maximum, quadratic variation, etc.).

 In Theorem \ref{NiceCompact} we  make it a point to show that the space of \emph{\cadlag} martingales corresponding to $(\mu_t)_{t\in [0,1]}$ carries a compact Polish topology. We then note that the Root solution of the Skorokhod problem yields an explicit Lipschitz-Markov kernel, establishing the existence of a Markovian martingale with prescribed marginals.
 
%
%

%

\subsection{Further literature}
 
 Lowther \cite{Lo07,Lo08b} is particularly interested in martingales which have a property even stronger than being Lipschitz Markov: He  shows that there exists a unique almost continuous diffusion martingale  whose marginals fit the given peacock. Under additional conditions on the peacock he is able to show that this martingale has (a.s.) continuous paths.

Hirsch-Roynette-Profeta-Yor \cite{HiPr11,HiRo12} avoid constructing Lipschitz-Markov-kernels explicitly. Rather they establish the link to the works of Gy\"ongy  \cite{Gy86} and Dupire \cite{Du94} on mimicking process / local volatility models,  showing  that Lipschitz-Markov martingales exist for sufficiently regular peacocks. This is extended to general peacocks through approximation arguments. 
On a technical level, their arguments differ from Kellerer's approach in that ultrafilters rather than compactness arguments are used to pass to accumulation points. 
We also recommend \cite{HiRo12} for a more detailed review of existing  results.

\section{The compact set of martingales associated to a peacock}\label{CompactnessSection}

It is well known and in fact a simple consequence of Prohorov's Theorem that $\Pi(\mu_1, \mu_2) $ is compact wrt\  the weak topology induced by the bounded continuous functions (see e.g.\ \cite[Section 4]{Vi09} for details).  It is also straightforward  that  the continuous functions $f:\R^2 \to \R$ which are bounded in the sense that $|f(x,y)|\leq \phi(x)+\psi(y)$ for some $\phi\in L^1(\mu_1), \psi\in L^1(\mu_2)$ induce the same topology on $\Pi(\mu_1, \mu_2)$. 

A transport plan  $\pi\in \Pi(\mu_1, \mu_2)$ is a martingale measure iff for all continuous, compact support functions $h$, $\int h(x)(y-x)\, d\pi=0$. Hence, $\M(\mu_1, \mu_2)$ is a closed subset of $\Pi(\mu_1, \mu_2)$ and thus compact.
    Likewise, $\M(\mu_1,\ldots, \mu_n)$ is compact. 
\subsection{The countable case}  

We fix a countable set $Q\ni 1$ which is dense in $[0,1]$ and write $\M_Q$ for the set of all martingale measures on $\R^Q$. 
 For $D\subseteq Q$ we set:
\[\M_Q((\mu_t)_{t\in D}):=  \{ \P \in\M_Q:  S_t\sim_{\P} \mu_t \mbox{ for $t\in D$} \}.\]
We equip $\R^Q$ with the product topology and consider $\M_Q$ with 
the topology of weak convergence with respect to continuous bounded functions. 
Note that this topology is in  fact induced by the functions $\omega \mapsto f(S_{t_1}(\omega), \ldots,S_{t_n}(\omega))$, where $t_i\in Q$ and $f$ is continuous and bounded. 

\begin{lemma}\label{lem:Mcompact} 
For every finite $D\subseteq Q, D\ni 1$ the set $\M_Q ((\mu_t)_{t\in D})$ is non-empty and compact. 
As a consequence,  $ \M ((\mu_t)_{t\in Q})=\M_Q((\mu_t)_{t\in Q})$  is non-empty and compact.
\end{lemma}
  
\begin{proof}  
We first show that $\M_Q(\mu_1)$ is compact.
To this end, we note that for every $\eps >0$ there exists $n$ such that $\int (|x|-n)_+\, d\mu_1 < \eps$. We then also have $$\textstyle \mu(\R\setminus [-(n+1), (n+1)])\leq \int (|x|-n)_+\, d\mu\leq \int (|x|-n)_+\, d\mu_1 < \eps$$ for every $\mu\preceq\mu_1$.

For every $r:Q\to \R_+$ the set $K_r:=\{g:Q\to\R, |g|\leq r\}$ is compact by 
Tychonoff's theorem. Also, for given $\eps>0$ there exists $r$ such that for all $\P$ on $\R^Q$ with $\law_\P (S_t)\preceq \mu_1$ for all $t\in Q$ we have $\P(K_r)>1-\eps$. 
Hence Prohoroff's Theorem implies that 
$\M_Q(\mu_1)$ is compact. 

Next observe that for any finite set $D\subset Q, 1\in D$ the set $\M_Q((\mu_t)_{t\in D})$ is non empty by Strassen's theorem. 
Clearly $\M_Q((\mu_t)_{t\in D})$ is also closed and hence compact. 
The family of all such sets $\M_Q((\mu_t)_{t\in D})$ has the finite intersection property, hence by compactness
$$\qquad\M_Q((\mu_t)_{t\in Q})=\textstyle \bigcap_{D\subseteq Q, 1\in D, |D|< \infty} M_Q((\mu_t)_{t\in D})\neq \emptyset. \qquad \qedhere$$ 
\end{proof}



 \subsection{The right-continuous case} 
  We will now extend this construction to right-continuous  families of marginals on the whole interval $[0,1]$.
 
 We first note that it is not necessary to distinguish between the terms right-continuous and \cadlag\ in this context: fix a (not necessarily countable) set $Q\subseteq [0,1], Q\ni 1$,   a peacock $(\mu_t)_{t\in Q}$ and a strictly convex function $\phi$ which grows at most linearly, e.g.\ $\phi(x)=\sqrt{1+x^2}$. Then the following is  straightforward:  
 the mapping  $\mu_{\cdot}:Q\to P(\R), q \mapsto \mu_q $  is \cadlag\ wrt\
the weak topology  on $P(\R)$ iff the increasing function $q\mapsto \int \phi\, d\mu_q$ is right-continuous. In this case we say that $(\mu_t)_{t\in Q}$ is a right-continuous peacock.

  As we have to deal with right limits we will recall the following:
 
\begin{lemma}\label{lem:backmart}
Let $( X_n)_{n\in -\mathbb{N} \cup \{-\infty\} }$  
be a martingale wrt $ (\G_n)_{n\in-\N \cup\{-\infty\}}$ and write 
$\mu_n = \law(X_n)$. If $\lim_{n \to -\infty}  \mu_n  =  \mu_{-\infty}$, 
then  $X _{-\infty}  =\lim  X_n$ a.s.\ and in $L_1$.
\end{lemma}

\begin{proof}
Set $Y := \lim_{n\to -\infty} X_n$  which exists (see for instance 
\cite[Theorem II.2. 	3]{ReYo99}), has the same law as  
$X_{-\infty }$ and satisfies $\E [Y| X_{-\infty} ]=  X_{-\infty}$ .
This clearly implies that  $X_{-\infty}  =  Y$.
\end{proof}    
As above,  we fix a countable and dense set $Q\subset [0,1]$ with $1\in Q$ and consider     
\begin{align*} 
\mathcal D &=  \{ g  :[0,1]  \to \R : g  \mbox{ is  \cadlag\ } \} ,\\
\D_Q  &=  \{ f: Q  \to \R : \exists g\in \mathcal D \,  \text{ s.t.\ } g_{|Q}  =  f \}.
\end{align*}
Note that  $\D_Q$  is a Borel subset of  $\R^Q$. Indeed a useful explicit description
of  $\D_Q$  can be given in terms of upcrossings. For  $f: Q  \to  \R$  we write 
$UP (f, [a,b ])$ for
the number of upcrossings of  $f$  through the interval $[ a,b ]$. Then  $f  \in   \D_Q$  iff $  f$
is \cadlag\ and bounded on $Q$  and satisfies  $UP (f, [a,b ]) < \infty $ for arbitrary $a < b$ 
(clearly it is enough to take $a,b \in Q$).
We also set \begin{align}\label{RCFilt} \bar\F_s  := \textstyle \bigcap_{t\in Q, t>s}   \F_t\end{align}
for $s \in [0,1)$ and let  $\bar\F_1  =  \F_1$.

\begin{proposition}\label{OnDQ}  
Assume that  $(\mu_t)_{t\in Q}$ is a right-continuous peacock and let $\P \in \M((\mu_t)_{t\in Q})$. Then $\P(\D_Q) = 1$. For $ q\in Q$, 
$\bar S_q:=S_q=\lim_{t\downarrow q, t\in Q, t>q} S_t$ holds $\P$-a.s. For  $s \in  [0 ,1]\setminus Q$,
$\lim_{t\downarrow s, t\in Q, t>s} S_t$ exists and we define it to be $\bar S_s$. 
The thus defined process $( \bar S_t)_{t \in [0,1]}$ is a \cadlag\ martingale wrt\ $(\bar\F_t) _{t\in [0,1]}$. 
\end{proposition}
 
\begin{proof} 
 By Lemma \ref{lem:backmart},  $S_q  =\lim_{t\downarrow q,t>q,t\in Q} S_t$  for all $ q  \in  Q$.
 Using standard martingale folklore (cf. \cite[Theorem 2.8]{ReYo99}), this implies that 
 $(S_t)_{t\in Q}$  is a martingale under $\pi$
 wrt\ $(\bar\F_t)_{t\in Q}$ as well and the paths of $(S_t)_{t\in Q}$ are almost surely \cadlag. 
 Moreover these are almost surely bounded by Doob's maximal inequality and have 
 only finitely many upcrossings by Doob's upcrossing inequality. This proves 
 $\P(\D_Q )=1$.  As the paths of $(S_t)_{t\in Q} $ are \cadlag\, the definition
 $\bar S_s:=\lim_{t\downarrow s, t\in Q, t>s}S_t$ is well for  
 $s \in [0, 1] \setminus Q$ and  $(\bar S_t )_{t\in[0,1]}$ 
 is a \cadlag\ martingale under $\P$ wrt\ $(\bar\F_t)_{t\in [0,1]}$.  
\end{proof}


 Identifying elements of $\mathcal D$ and $\D_Q$,  the right-continuous filtration $\F^+$ on $\mathcal D$ equals  the restriction of   $\bar \F$ (cf.\ \eqref{RCFilt}) to $\D_Q$. Since any martingale  measure $\P$  
concentrated on $\D_Q$ corresponds to a martingale measure $\widetilde \P$ on 
$\mathcal D$ Proposition \ref{OnDQ}  yields: 
\begin{proposition}\label{Q2I} Let $(\mu_t)_{t\in [0,1]}$ be a right-continuous peacock and $Q\ni 1, Q\subseteq [0,1]$ a countable dense set. Then the above correspondence 
\begin{align}\label{QtoI} \P\mapsto \widetilde \P\end{align}
 constitutes a bijection between $\M((\mu_t)_{t\in Q})$ and $\M((\mu_t)_{t\in [0,1]})$.
\end{proposition}
Through the identification $\P \mapsto \widetilde \P$, 
the set $\M((\mu_t)_{t\in [0,1]})$ carries a compact topology. Superficially, this topology seems to depend on the particular choice of the set $Q$ but in fact this is not the case: indeed given $Q, Q'$ the set $Q\cup Q'$ gives rise to a topology which is a priori finer than the ones corresponding to $Q$ resp.\ $Q'$. But as all involved topologies are compact, they are in fact equal. Hence we obtain:
\begin{theorem}\label{NiceCompact} Let $(\mu_t)_{t\in [0,1]}$ be  a right-continuous peacock and consider the canonical process $(S_t)_{t\in [0,1]}$ on the Skorokhod space $\mathcal D$.
The set $\M((\mu_t)_{t\in [0,1]}) $ of martingale measures with marginals $(\mu_t)$  is non empty and compact wrt\  the topology induced by the functions $$\omega \mapsto f(S_{t_1}(\omega), \ldots, S_{t_n}(\omega)),$$ 
where $t_1, \ldots, t_n\in [0,1]$ and $f$ is continuous and bounded.
\end{theorem}
%
%
%

\subsection{General peacocks} Kellerer  \cite{Ke72} considers the more general case of a peacock $(\mu_t)_{t\in T}$ where $(T,<)$ is an abstract total order and $s < t$ implies $\mu_s \preceq  \mu_t$, moreover no continuity assumptions on $t\mapsto \mu_t$ are imposed. Notably the existence of a martingale associated to such a general peacock already follows from the case treated in the previous section since every peacock can be embedded in a (right-) continuous peacock indexed by real numbers: 
\begin{lemma}\label{NotRight} Let $(T, <)$ be a total order and 
$(\mu_t)_{t\in T}$ a peacock. Then there exist a  peacock $(\nu_s)_{s\in \R^+}$ which is continuous (in the sense that $s\mapsto \nu_s$ is weakly continuous) and an increasing function $f:T\to \R_+$ such that $$\mu_t=\nu_{f(t)}.$$ If $T$ has a maximal element we may assume that $f:T\to [0,1]$.
\end{lemma} 
\begin{proof} Assume first that $T$ contains a maximal element $t^*$. Consider again $\phi(x)=\sqrt{1+x^2}$ and set $f(t):=\int \phi\, d \mu_t$ for $t\in T$. On the image $S$ of $f$ we define $(\nu_s)$ through $\nu_{f(t)}:= \mu_t$. Then $s\mapsto \nu_s$ is continuous on $I$ and $s^*:=f(t^*)$ is a maximal element of $S$. 

Using tightness of $(\nu_s)_{s\in S}$ we obtain that 
 $\nu_s:=\lim_{r\in S, r\to s}$  exists for $s\in \overline S$. It remains to extend $(\nu_s)_{s\in \overline S}$ to $[0,s]$. The set $[0,s]\setminus S$ is the union of countably many intervals and on each of these we can define $\nu_s$ by linear interpolation. Finally it is of course possible to replace $[0,s]$ by $[0,1]$ through rescaling.

If $T$ does not have a maximal element, we first pick an increasing sequence $(t_n)_{n\geq 1}$ in $T$ such that $\sup_n \int \phi\, d \mu_{t_n}= \sup_{t\in T}\int \phi\, d \mu_t$, then we apply the previous argument to the initial segments $\{ s\in T:s\leq t_n\}$.
\end{proof}

Above we have seen that $\M((\mu_t)_{t\in [0,1]})\neq \emptyset$ for  $(\mu_t)_{t\in [0,1]}$ right-continuous and pasting countably many martingales together this extends to the case of a right-continuous peacock $(\nu_s)_{s\in \R_+}$. By Lemma \ref{NotRight} this already implies $\M((\mu_t)_{t\in T})\neq \emptyset$ for a peackock wrt to a general total order $T$.


\section{Root to Markov}\label{RootSec}

So far we have constructed martingales which are not necessarily Markov. To obtain the existence of a Markov-martingale with desired marginals, one might try  to adapt the previous argument by restricting the sets $\M_Q((\mu_t)_{t\in D})$ to the set of Markov-martingales. As noted above, this strategy does not work in a completely straight forward way as  being  \emph{Markovian}  is not a closed property wrt\ weak convergence.

\begin{example}
The sequence $\mu_n = \frac{1}{2}(\delta_{(1,\frac{1}{n},1)}+\delta_{(-1,-\frac{1}{n},-1)})$
of Markov-measures weakly converge to the non-Markovian measure 
$\mu=\frac{1}{2}(\delta_{(1,0,1)}+\delta_{(-1,0,-1)})$.
\end{example}
%
%
%
%

\subsection{Lipschitz-Markov kernels}\label{RootSec}

A solution $\tau$ to the two marginal Skorokhod problem $B_0\sim \mu, B_\tau\sim \nu$  gives rise to the particular martingale transport plan $(B_0, B_\tau)$. Sometimes these martingale couplings induced by solutions to the Skorokhod embedding problem exhibit certain desirable properties. In particular we shall be interested in the Root solution to the Skorokhod problem.   
\begin{theorem}[Root \cite{Ro69}]\label{thm:root} 
Let $\mu \preceq \nu$ be two probability measures on $\mathbb{R}$. There
exists a closed set (``barrier'') $\mathcal{R}\subset \R_+\times \R$ (i.e.\ $(s,x)\in \mathcal R, s< t$ implies that $(t,x)\in \mathcal R$)  such that for Brownian motion $(B_t)_{t \geq 0}$ started in
$B_0 \sim \mu$ the hitting time $\tau_R$ of $\mathcal{R}$ embeds $\nu$ in the sense that $B_{\tau_R} \sim \nu$ and $(B_{t\wedge \tau_R})_t$ is uniformly integrable.
\end{theorem}


Before we formally introduce the Lipschitz-Markov property we recall that the $L^1$- Wasserstein distance between two probabilities $\alpha, \beta$ on $\R$  is given by 
$$\textstyle W(\alpha,\beta)=\inf\Big\{\int |x-y| \, d\gamma: \gamma\in \Pi(\alpha, \beta)\Big\}=
\sup\Big\{\int f\, d\nu-\int f\, d\mu: f\in \mbox{Lip}_1\Big\},$$
where $\Pi(\alpha,\beta)$ denotes the set of all couplings between $\alpha$ and $\beta$ and $\mbox{Lip}_1$ denotes the set of all $1$- Lipschitz functions  $\R\to \R.$ 
The equality of the two terms is a consequence of the Monge-Kantorovich duality in optimal transport, see e.g.\ \cite[Section 5]{Vi09}. 

A martingale coupling $\pi\in\M(\mu,\nu)$ is \emph{Lipschitz-Markov} iff for some (and then any) disintegration $(\pi_x)_x$ of $\pi$ wrt\ $\mu$  and some set $X\subseteq \R$, $\mu(X)=1$ we have for  $x,x'\in X$ 
\begin{align}\label{def:LM} \Wa(\pi_x,\pi_{x'})= |x-x'|.\end{align}
We note that the inequality $\Wa(\pi_x,\pi_{x'})\geq  |x-x'|$ is satisfied for arbitrary $\pi\in\M(\mu,\nu)$: 
 for typical  $x,x', x< x'$, the mean of $\pi_x$ equals $x$ and the mean of $\pi_{x'}$ equals $x'$. We thus find for arbitrary $\gamma\in \Pi(\pi_x,\pi_{x'})$
\begin{align}\label{ShiftCost}\textstyle
\int |y-y'|\, d\gamma(y,y')\geq &\, \textstyle \big|\int y\, d\gamma(y,y')-\int y'\, d\gamma(y,y')\big|\\
= &\, \textstyle \big|\int y\, d\pi_x(y)-\int y'\, d\pi_{x'}(y')\big|=|x-x'|, \nonumber
\end{align}
hence $W(\pi_x,\pi_{x'})\geq |x-x'|$. 

Note also  that  $W(\pi_x,\pi_{x'})= |x-x'|$ holds iff the inequality in \eqref{ShiftCost} is an equality for the minimizing coupling $\gamma^*$. This holds  true iff there is a transport plan $\gamma$  which is \emph{isotone} in the sense that it transports $\pi_x$-almost all points $y$ to some $y'\geq y$. This is of course equivalent  to saying that $\pi_x$ precedes  $\pi_{x'}$ in first order stochastic dominance. 

\begin{lemma}\label{lem:RootLM}
The Root coupling $\pi_R= \mbox{Law}(B_0,B_{\tau_R})$ is Lipschitz-Markov.
\end{lemma}
\begin{proof} Write $(B_t)_t$ for the canonical process on $\Omega=C[0,\infty)$, $\W$ for Wiener measure started in $\mu$ and $\tau_R$ for the Root stopping time s.t.\ $(B_0, B_{\tau_R})\sim_\W \pi_R\in \M(\mu, \nu)$. 

It follows from the 
 geometric properties of the barrier $\mathcal R$ that for all $x<x'$ and $\omega \in \Omega$ such that $\omega(0)=0$ 
$$ B_{\tau_R(x+\omega)}(x+\omega)\leq B_{\tau_R(x'+\omega)}(x'+\omega).$$
Write $\pi_x$ for the distribution of $B_{\tau_R}$ given $B_0=x$ and $\W_0$ for Wiener measure with start in $0$. Then $(\pi_x)_x$ defines a disintegration (wrt\ the first coordinate) of $\pi_R$ and  for $x<x'$ an isotone coupling $\gamma\in\Pi(\pi_x,\pi_{x'})$ can be explicitly defined by
$$ \qquad\gamma(A\times B):= \textstyle \int {1}_{A\times B}(B_{\tau_R(x+\omega)}(x+\omega),B_{\tau_R(x'+\omega)}(x'+\omega))\, \W_0(d\omega).  \qquad\qedhere $$
\end{proof}

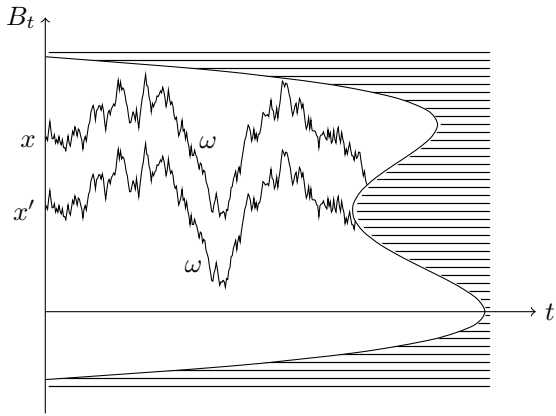
\begin{center}
\begin{tikzpicture}[scale=1.5]
\draw[->] (0,-1) -- (0,2.5) node[left] {$B_t$};
\draw [smooth,variable=\y,domain=-0.7:2.1536] plot ( {3.8-\y-4*\y*\y+6*\y*\y*\y-2*\y*\y*\y*\y},{\y});
\fill[pattern=horizontal lines, inner sep=2pt] (0.03,-0.8) -- (0.03,-0.7125) -- plot [domain=-0.7:2.1536] ( {3.83-\x-4*\x*\x+6*\x*\x*\x-2*\x*\x*\x*\x},{\x}) -- (0.03,2.2) -- (3.9,2.2) -- (3.9,-0.8) -- cycle;
\draw[->] (0,-0.1) -- (4.3,-0.1) node[right] {$t$};
\pgfmathsetseed{612}
\draw (0,1.41) node [left] {$x$}
\foreach \x in {1,...,281} { -- ++(0.01,-rand*0.1) } -- ++(0.01,-0.04) node[right] {};
\pgfmathsetseed{612}
\draw (0,0.8) node [left] {$x'$}
\foreach \x in {1,...,272} { -- ++(0.01,-rand*0.1) } node[right] {};
\node at (1.3,0.3) {$\omega$};
\node at (1.43,1.4) {$\omega$};
\end{tikzpicture}
\end{center}

\begin{remark} We thank David Hobson for pointing out that 
 Lemma \ref{lem:RootLM} remains true if we replace $\tau_R$ by  Hobson's solution to the Skorokhod problem \cite{Ho98b}.\footnote{Hobson's solution \cite{Ho98b} can be seen as an extension of the Azema-Yor embedding to the case of a general starting distribution.}
 
 We also note that this property is not common among martingale couplings. It is not present  e.g.\ in the coupling corresponding to the Rost-embedding nor  the various  extremal martingale couplings recently introduced by  Hobson--Neuberger \cite{HoNe12}, Hobson--Klimmek \cite{HoKl12},  Juillet (and one of the present authors) \cite{BeJu12},   and Henry-Labordere--Touzi \cite{HeTo13}.
\end{remark}

\subsection{Compactness of Lipschitz-Markov martingales}

To generalize the Lipschitz-Markov property to multiple time steps we first provide an equivalent formulation in the two step case. Using the Lipschitz-function characterization of the Wasserstein distance we find that \eqref{def:LM} is tantamount to the following: 
for every $f\in \mbox{Lip}_1(\R)$ the mapping
\begin{align}\label{Lip2}
\textstyle x\mapsto \int f\, d\pi_x = \E[f(S_2)| S_1=x]
\end{align}
is $1$-Lipschitz (on a set of full $\mu$-measure). 

\smallskip

Let $Q\subset [0,1]$ be  a  set which is at most countable. In accordance with \eqref{Lip2} we call a measure/coupling $\P$ on $\R^Q$  \emph{Lipschitz-Markov} if for any $s,t \in Q, s < t$ and 
$f \in \mbox{Lip}_1(\R)$ there exists
$g \in \mbox{Lip}_1(\R)$ such that 
\begin{align}\label{eq:LM}
\mathbb{E}_\P[f(S_{t})| \F_s] = g(S_s).
\end{align}

The Lipschitz-Markov property is closed in the desired sense: 
\begin{lemma}\label{lem:lipclosed}
A martingale measure $\P$ on $\R^Q$ is Lipschitz-Markov iff \begin{align}
\textstyle
\E_\P[X f(S_{t})]\,\E_\P[Y]
-\E_\P[X]\, \E_\P[Y f(S_{t})]\leq \int X(\omega)Y(\bar\omega)|\omega_{s}-\bar\omega_{s}|\, d(\P\otimes \P)
\label{eq:Lq}
\end{align}
for all $f \in \mbox{Lip}_1(\R)$, $s < t \in Q$ and $X,Y$ non-negative, bounded, and $\F_s$-measurable. 
\end{lemma}

\begin{proof}
If $\P$ is Lipschitz-Markov, then for a given 1-Lipschitz function $f$ we can find by definition of a Lipschitz-Markov measure/coupling a 1-Lipschitz function $g$ satisfying \eqref{eq:LM}. Moreover, as $g \in \mbox{Lip}_1$ we have for non-negative, bounded $X,Y$
\[ (g(\omega_{s})-g(\bar\omega_{s}))X(\omega)Y(\bar\omega)\leq |\omega_{s}-\bar\omega_{s}|X(\omega)Y(\bar\omega).\]
Integration with respect to $\P \otimes \P$ and an application of \eqref{eq:LM} yields \eqref{eq:Lq}.

For the reverse implication, by basic properties of conditional expectation there is a $\sigma((S_q)_{q \in Q \cap [0,s]})$-measurable function $\psi$ such that $\P$-a.s.\
\[\psi(\omega) = \mathbb{E}_\P[f(S_{t})| \F_s](\omega) .\]
Now from \eqref{eq:Lq} we almost surely have
$ \psi(\omega)-\psi(\bar\omega) \leq |\omega_{s}-\bar\omega_{s}|$
which shows that $\psi$ only depends on the $s$ coordinate and is in fact 1-Lipschitz.
\end{proof}

For $D\subseteq Q$ we set
\[\L_Q((\mu_t)_{t\in D}):=\{ \P \in \M_Q : \P \mbox{ is Lipschitz-Markov,   $S_t\sim_\P \mu_t$ for $t\in D$} \}.\]

\begin{theorem}\label{lem:markcomp} Let $Q \subseteq [0,1], Q\ni 1$ be countable.
 For every finite $ 1 \in D\subseteq Q$ the set $\L_Q((\mu_t)_{t\in D})$ is non-empty and compact. 
In particular, $\L((\mu_t)_{t \in Q}) := \L_Q((\mu_t)_{t\in Q})$ is non-empty and compact.
\end{theorem}
\begin{proof}
For finite $D\subseteq Q$ it is plain that $\L_Q((\mu_t)_{t\in D})$ is non-empty: this follows by composing of Lipschitz-Markov-kernels. Hence, by compactness, $\L_Q((\mu_t)_{t \in Q}) = \bigcap_{D \subseteq Q, |D| < \infty} \L_Q((\mu_t)_{t\in D})\neq \emptyset$. \end{proof}
 A martingale on $\mathcal{D}$ is \emph{Lipschitz-Markov} if
\eqref{eq:LM} holds for $s<t\in[0,1] $  wrt\ $\F^+$.
\begin{theorem}\label{OurKellerer}
Assume that $(\mu_t)_{t\in [0,1]}$ is a right-continuous peacock and let $Q\ni 1$ be countable and dense in $[0,1]$. 
If $\P\in \L((\mu_t)_{t\in Q})$, then the corresponding (cf.\eqref{QtoI}) martingale measure  $\widetilde \P\in \M((\mu_t)_{t\in [0,1]})$ is Lipschitz-Markov. 

In particular, the set of all Lipschitz-Markov martingales with marginals $(\mu_t)_{t\in [0,1]}$ is compact and non-empty.
\end{theorem} 
\begin{proof}
The arguments in the proof of Lemma \ref{lem:lipclosed} work in exactly the
same way to show that $\tilde \P$ being Lipschitz-Markov is equivalent to conditions similar to
\eqref{eq:Lq} where $X,Y$ are chosen to be measurable wrt\ $\F^+_{s}$
(or $\bar \F_s$, see the remark before Proposition \ref{Q2I}).

For arbitrary $s,t \in [0,1], s<t$ choose sequences $s_n\downarrow s, t_n \downarrow t$ in
$Q$. Note that $X,Y$ are in fact measurable wrt\ $\F_{s_n}$ and we thus have
\begin{align}\nonumber\textstyle
\E_\P[X f(S_{t_n})]\,\E_\P[Y]
-\E_\P[X]\, \E_\P[Y f(S_{t_n})]\leq \int X(\omega)Y(\bar\omega)|\omega_{s_n}-\bar\omega_{s_n}|\, d(\P\otimes \P)(\omega,\bar\omega)
\end{align} by Lemma \ref{lem:lipclosed}. Letting $n\to \infty$ concludes the proof.
\end{proof}

\subsection{Further comments}
It is plain  that a Lipschitz-Markov kernel also has the Feller-property and
 in particular a Lipschitz-Markov martingales  are strong Markov processes wrt\ $\F^+$ (see \cite[Remark 1.70]{Li10}).
As in the previous section, the right-continuity of $(\mu_t)_{t\in [0,1]}$ is not necessary to establish the existence of a Lipschitz-Markov martingale,   this follows from Lemma \ref{NotRight}.
We also remark that the arguments of Section \ref{CompactnessSection}  directly extend to the case of multidimensional peacocks, where the marginal distributions $\mu_t$ are probabilities on $\R^d$. However it remains open whether Theorem \ref{OurKellerer} extends to this multidimensional setup.



\bibliography{joint_biblio}{}
\bibliographystyle{plain}

\end{document}